\documentclass[11pt]{article}

\usepackage{amsmath,amsthm,amssymb,amsfonts, graphicx}
\usepackage{graphics}
\usepackage{authblk}
\usepackage{upgreek}
\usepackage{amsrefs,hyperref}
\theoremstyle{plain}
\newtheorem{theorem}{Theorem}[section]

\newtheorem{corollary}[theorem]{Corollary}
\newtheorem{definition}[theorem]{Definition}

\newtheorem{lemma}[theorem]{Lemma}
\newtheorem{proposition}[theorem]{Proposition}

\theoremstyle{definition}
\newtheorem{remark}[theorem]{Remark}
\numberwithin{equation}{section}
\newcommand{\diff}{\mathop{}\!\mathrm{d}}
\DeclareMathOperator{\Hess}{Hess}
\DeclareMathOperator{\tr}{tr}

\DeclareMathOperator{\supp}{supp}

\title{Finite-time blowup for the inviscid vortex stretching equation}
\author[1]{Evan Miller}
\affil[1]{University of British Columbia, Department of Mathematics

emiller@msri.org}

\begin{document}

\maketitle

\begin{abstract}
In this paper, we will introduce the inviscid vortex stretching equation, which is a model equation for the 3D Euler equation where the advection of vorticity is neglected. We will show that there are smooth solutions of this equation which blowup in finite-time, even when restricting to axisymmetric, swirl-free solutions. This provides further evidence of the role of advection in depleting nonlinear vortex stretching for solutions of the 3D Euler equation.
\end{abstract}

\section{Introduction}

The Euler equation is one of the fundamental equations of fluid dynamics.
The incompressible Euler equation describes the flow of a fluid at constant density and with no viscosity; it can be expressed by
\begin{align} \label{Euler}
    \partial_t u 
    +(u\cdot\nabla)u
    +\nabla p
    &=0 \\
    \nabla\cdot u
    &=0,
\end{align}
where $u$ is the velocity, and the pressure, $p$, is determined completely by $u$.
Another important vector field for three dimensional flow is the vorticity, $\omega=\nabla \times u$. The evolution equation for the vorticity is given by
\begin{align} \label{EulerVort}
    \partial_t\omega
    +(u\cdot\nabla)\omega
    -(\omega\cdot \nabla)u
    &= 0 \\
    u =
    \nabla\times(-\Delta)^{-1}&\omega.
\end{align}
These formulations are equivalent: it is a straightforward computation to show that $u$ is a solution of \eqref{Euler} if and only if $\omega$ is a solution of $\eqref{EulerVort}$.
The term $(u\cdot\nabla)\omega$ gives rise to advection, while the term $(\omega\cdot\nabla)u$ gives rise to vortex stretching.

We will consider a model equation for the three dimensional Euler equation that focuses on the effect of vortex stretching by neglecting the advection term. This yields the following evolution equation, which we will call the inviscid vortex stretching equation:
\begin{align}
    \partial_t \omega &=
    P_{df}\left((\omega\cdot\nabla)u\right) \\
    u&=\nabla\times
    (-\Delta)^{-1}\omega.
\end{align}
Note that we must take the Helmholtz projection onto divergence free vector fields $P_{df}$ in order to preserve the divergence free structure of the equation. 
This equation can also be expressed in terms of the symmetric part of the velocity gradient 
$S_{ij}=\frac{1}{2}
    (\partial_i u_j+ \partial_j u_i)$, as
\begin{equation}
    \partial_t\omega=P_{df}(S\omega).
\end{equation}
For sufficiently smooth initial data, this equation has strong solutions locally in time.

\begin{theorem} \label{WellPosedIntro}
For all $\omega^0\in H^s_{df}\left(\mathbb{R}^3\right), s>\frac{3}{2}$,
there exists a unique strong solution of the inviscid vortex stretching equation
$\omega\in C^1\left(\left[0,T_{max}\right);
H^s_{df}\left(\mathbb{R}^3\right)\right)$, where \begin{equation}
T_{max}
\geq 
\frac{1}{2C_s\left\|\omega^0\right\|_{H^s}}.
\end{equation}
Furthermore, if $T_{max}<+\infty,$ then for all $0\leq t<T_{max}$,
\begin{equation}
    \|\omega(\cdot,t)\|_{H^s}
    \geq
    \frac{1}{2C_s (T_{max}-t)}.
\end{equation}
\end{theorem}

There are smooth solutions of the inviscid vortex stretching equation that blowup in finite-time.
This is true even in geometric setups, such as axisymmetric, swirl-free flows, where the three dimensional Euler equation has global-in-time smooth solutions.

\begin{theorem} \label{BlowupThmIntro}
Suppose $\omega^0\in H^s_*\left(\mathbb{R}^3\right), s>\frac{3}{2}$ is the vorticity for an axisymmetric, swirl free vector field, is not identically zero, and further that $\omega^0_{\theta}(r,z)$ is odd in $z$, compactly supported away from the axis $r=0$ and the plane $z=0$, 
and for all $r,z\geq 0$,
\begin{equation}
    \omega^0_\theta(r,z)\leq 0.
\end{equation}
Then the solution of the inviscid vortex stretching equation $\omega \in
C^1\left([0,T_{max});H^s_*
\left(\mathbb{R}^3\right)\right)$ with this initial data preserves these properties 
and blows up in finite-time
\begin{equation}
    T_{max}\leq 
    \frac{1}{\kappa Q^0},
\end{equation}
where the constant $\kappa>0$ depends on the geometry of $\Omega^0$, the support of the vorticity, and
\begin{equation}
    Q^0
    =
    -\int_0^\infty\int_0^\infty
    r^2\omega^0_{\theta}(r,z)
    \diff r \diff z.
\end{equation}
\end{theorem}

\subsection{Discussion of results}
The local wellposedness result in Theorem \ref{WellPosedIntro} is the same regularity class required for smooth solutions of the 3D Euler equation.
For initial data $u^0\in H^s_{df}\left(\mathbb{R}^3\right), s>\frac{5}{2}$, Kato proved the local-in-time well-posedness of solutions of the Euler equation \cite{KatoLWP}.
We will note this is equivalent to local-well-posedness for initial data $\omega^0 \in \dot{H}^{-1} \cap H^s, s>\frac{3}{2}$. We do not require that $u^0\in L^2$---or equivalently that $\omega^0\in \dot{H}^{-1}$---because there is no energy equality for the inviscid vortex stretching equation. One of the fundamental properties of the Euler equation is that for all $0<t<T_{max}$,
\begin{equation} \label{EnergyEquality}
    \|u(\cdot,t)\|_{L^2}^2=
    \left\|u^0\right\|_{L^2}^2.
\end{equation}
The change in the structure of the equation that comes from dropping the advection term means that the the energy equality \eqref{EnergyEquality} does not hold for the inviscid vortex stretching equation.

It is also unclear whether the Beale-Kato-Majda criterion will hold for the inviscid vortex stretching equation. The Beale-Kato-Majda criterion \cite{BealeKatoMajda} states that if a smooth solution of the Euler equation blows up in finite-time, then
\begin{equation}
    \int_0^{T_{max}}
    \|\omega(\cdot,t)\|_{L^\infty} 
    \diff t
    =+\infty.
\end{equation}

The finite-time blowup result for the inviscid vortex stretching equation, Theorem \ref{BlowupThmIntro}, shows conclusively that the only factor that can prevent blowup by nonlinear vortex stretching is depletion of nonlinearity by the advection. This has already been observed in some one dimensional models, as discussed in detail in \cite{ElgindiJeong}.
Constantin, Lax, and Majda introduced a one dimensional model for the vorticity equation and showed that this model equation blows up in finite-time \cite{ConstantinLaxMajda}.
The Constantin-Lax-Majda model is given by
\begin{equation}
    \partial_t \omega =\omega H[\omega],
\end{equation}
where $H$ is the Hilbert transform.
This model equation was modified by De Gregorio \cite{DeGregorio} to include the impact of advection, giving the evolution equation,
\begin{align}
    \partial_t\omega
    +u\partial_x\omega 
    -\omega \partial_x u
    &= 0 \\
    \partial_x u&=H[\omega].
\end{align}

This is a much more realistic one dimensional model, as both vortex stretching and advection play a role.
The De Gregorio model equation is also much less singular than the Constantin-Lax-Majda model equation, due to a depletion of nonlinearity by advection. For example, with initial data $\omega^0=\sin$, the Constantin-Lax-Majda model blows up in finite-time. For the De Gregorio model $\omega=\sin$ is a steady state, and in fact is stable for small perturbations \cite{JiaStewartSverak}.
Global regularity for smooth solutions of the De Gregorio model equation has been conjectured, although this has turned out not to be the case. Recently, Elgindi and Jeong have proven blowup for solutions of the De Gregorio equation that are only $C^\alpha$, and not smooth \cite{ElgindiJeong}. Chen, Hou, and Huang recently improved this result, proving finite-time blowup for the De Gregorio equation with smooth initial data \cite{ChenHouHuang}. This phenomenon has also been analyzed by Sarria and Saxton in the context of the inviscid Proudman-Johnson equation, which can be obtained from the Euler equation under a stagnation point Ansatz \cites{SarriaSaxton,SarriaSaxton2}.

This is very similar to the picture for axisymmetric, swirl-free solutions of the Euler equation in three dimensions.
Ukhovskii and Yudovich proved global regularity for axisymmetric, swirl-free solutions of the three dimensional Euler equation, assuming smooth enough initial data \cite{Yudovich}.
Danchin extended this global regularity result by weakening the conditions on the initial data to be broad enough to guarantee global regularity whenever the initial data is axisymmetric, swirl-free, and in $H^s\left(\mathbb{R}^3\right), s>\frac{5}{2}$, the standard regularity class for strong solutions of the Euler equation \cites{Danchin1,Danchin2}.
Very recently, Elgindi proved finite-time blowup for axisymmetric swirl-free solutions of the Euler equation with vorticity that is only $C^\alpha$, not smooth \cite{ElgindiBlowup}. Interestingly, the geometric setup for this blowup---an axisymmetric, swirl-free velocity, with $\omega_\theta(r,z)$ odd in $z$ and negative on the upper half plane---is precisely the same as in our blowup result Theorem \ref{BlowupThmIntro}. The primary difference is that, in Elgindi's result, blowup occurs at the axis of symmetry, whereas in Theorem \ref{BlowupThmIntro}, the vorticity is supported away from the axis of symmetry.

In addition to the various one dimensional models discussed above, there has also been work on three dimensional models for the vorticity equation. The regularizing role of advection was noted in the classic work of Constantin on the distorted Euler equation, which also neglect the advection of vorticity \cite{Constantin}. It should be noted that the model considered by Constantin has some fundamental differences with the model considered here. In particular, the Constantin's distorted Euler equations do not preserve the divergence free condition as the model under consideration in this paper does.
Constantin proves blowup for solutions of the distorted Euler equation having the form
\begin{equation}
    \omega(x)=f(|x|)\frac{x}{|x|},
\end{equation}
and this property---so long as $f$ is regular enough and vanishes at the origin---guarantees that $\omega$ will be a gradient. In particular, if $g'=f$, then
\begin{equation}
    \omega(x)=\nabla g(|x|)
\end{equation}

\begin{remark}
A little more should be said about the relationship of the model equation considered in this paper to the one introduced by Constantin in \cite{Constantin}.
Taking $U=\nabla u$, the distorted Euler equation is given by
\begin{equation}
    \partial_t U+ U^2 + \Hess(-\Delta)^{-1}\tr(U^2)
    =0.
\end{equation}
Decomposing the velocity gradient into symmetric and anti-symmetric parts, $U=S+A$, and using the fact that the anti-symmetric matrix can be represented using the vorticity vector as
\begin{equation}
    A=\frac{1}{2}\left(
    \begin{array}{ccc}
        0 & \omega_3 & -\omega_2  \\
       -\omega_3  & 0 & \omega_1 \\
       \omega_2 & -\omega_1 & 0
    \end{array}
    \right),
\end{equation}
this can be expressed as a coupled system of equations
\begin{align}
    \partial_t S+S^2+\frac{1}{4}\omega\otimes\omega
    -\frac{1}{4}|\omega|^2 I_3
    +\Hess(-\Delta)^{-1}\tr(U^2) &=0 \\
    \partial_t \omega -S\omega &=0.
\end{align}

The first clear difference between the model considered in this paper and the distorted Euler equation is that the distorted Euler equation does not have a Helmholtz projection applied to the vortex stretching term, but the differences go deeper. The distorted Euler equation---because it does not preserve the divergence free constraint---does not have a Biot-Savart law. In our equation, $S$ is determined completely by the vorticity with
\begin{equation}
    S=\nabla_{sym}\nabla\times (-\Delta)^{-1}\omega.
\end{equation}
In the distorted Euler equation, $S$ and $\omega$ cannot be determined in terms of each other, and both evolve via coupled evolution equations.
There are major similarities at a philosophical level between the blowup results considered in this paper and the blowup results for the distorted Euler equation in \cite{Constantin}: both show that advection tends to deplete the nonlinear vortex stretching. The equations are, however, substantially different at a technical level.
\end{remark}

In a series of papers, Hou and a number of co-authors consider the axisymmetric Euler equations with swirl.
These equations can be formulated as
\begin{align}
    (\partial_t+u_r\partial_r
    +u_z\partial_z)\frac{u_\theta}{r}
    &=
    -2u_r \frac{u_\theta}{r} \\
    (\partial_t+u_r\partial_r
    +u_z\partial_z)
    \frac{\omega_\theta}{r}
    &=
    \partial_z\left(\frac{u_\theta}{r}\right)^2,
\end{align}
where $u_r$ and $u_z$ can be determined from $\frac{\omega_\theta}{r}$ by the Biot-Savart law.
Hou and Lei \cite{HouLei} proposed a model equation where the advection terms $u_r\partial_r+u_z\partial_z$ are neglected, both in the context of the Euler and Navier--Stokes equations. Hou et al. proved finite-time blowup for these axisymmetric equations with vorticity neglected \cite{HouEtAl}.
See also \cites{HouLi,Hou2018,HouReg} for further work on this scenario.
In particular, Hou, Liu, and Wang proved global regularity for the viscous case of the these equations when the advection term is strengthened \cite{HouReg}.

We should note that the present work involves some fundamental differences with the approach of Hou and various co-authors. 
In particular, when the axisymmetric Euler equations are expressed in terms of $\frac{\omega_\theta}{r}$, a portion of the vortex stretching is absorbed into advection. Note in particular that
\begin{equation}
    u_r\partial_r \frac{\omega_\theta}{r}
    =
    \frac{1}{r}\left(
    u_r\partial_r \omega_\theta
    -\frac{u_r}{r}\omega_\theta\right).
\end{equation}
The latter of these two terms comes from vortex stretching $(\omega\cdot\nabla) u.$ 
While the work in this paper is in the same spirit as \cite{HouEtAl}, it is not a special case of the results in that paper. Hou et al. neglect the advection of $\frac{\omega_\theta}{r}$, while we neglect the advection of $\omega_\theta$.
It should also be noted that the inability to absorb vortex stretching into advection for vorticities that are only $C^\alpha$ was the key mechanism for Elgindi's $C^{1,\alpha}$ blowup solutions of the axisymmetric, swirl-free Euler equations.

\begin{remark}
The finite-time blowup for strong solutions of the inviscid vortex stretching equation is consistent with the large body of previous work involving both one and three dimensional models of the Euler equation in vorticity form. The key advance here is that the inviscid vortex stretching equation is an evolution equation on the relevant constraint space: the space of divergence free vector fields on $\mathbb{R}^3$. Because this is the appropriate constraint space for the three dimensional Euler equation, the geometric structure of these blowup solutions has a direct bearing on solutions of the actual Euler equation in a way that one dimensional models simply cannot. While our blowup Ansatz is axisymmetric and swirl-free, this model equation for inviscid fluids is also flexible enough to consider the general form, and is not specific to axisymmetric, swirl-free flows. There is every reason to expect that blowup also happens in more complicated geometries, but it would be harder to find monotone quantites in these settings.
\end{remark}

In fact, the blowup result is very directly related, even down to the functional used to prove blowup, to a very recent result of Choi and Jeong \cite{ChoiJeong}, which establishes that axisymmetric, swirl-free solutions of the Euler equation satisfying the geometric constraints of Theorem \ref{BlowupThmIntro} exhibit vorticity growth at infinity of the form
\begin{equation}
    \|\omega(\cdot,t)\|_{L^\infty}
    \geq C(1+t)^{\frac{1}{15}-\epsilon}.
\end{equation}
We will discuss this further in section \ref{SectionBlowupAS}. In particular, we will prove Theorem \ref{AnisoThm}, which requires that for solutions of the full Euler equation with initial data satisfying the geometric conditions in Theorem \ref{BlowupThmIntro}, the support of the vorticity must become sufficiently anisotropic to avoid finite-time blowup.

In addition to being the same geometric setup considered by Elgindi \cite{ElgindiBlowup} and Choi and Jeong \cite{ChoiJeong}, the geometric setup for blowup solutions of the inviscid vortex stretching equation in Theorem \ref{BlowupThmIntro} is related to the geometry of colliding vortex rings.
It is not necessary in Theorem \ref{BlowupThmIntro} that the vorticity be highly concentrated,
but when it is highly concentrated, this is precisely the setup of colliding vortex rings. 
In the case of the inviscid vortex stretching equation, due to the lack of advection, the vortex rings do not collide, but simply magnify each other.
Colliding vortex rings have a long history in the fluid mechanics literature both in terms of our understanding of vortex stretching and the onset of turbulence. In the 1970s, Oshima conducted the first major laboratory observation of colliding vortex rings \cite{Oshima}; they have remained a central example in the fluid mechanics literature ever since.

\subsection{Organization}
In section \ref{SectionDefinitions}, we will give definitions for a number of relevant spaces, state several key lemmas, and go over notation.
In section \ref{SectionLocalWP}, we will develop the local wellposedness theory for the inviscid vortex stretching equation and prove Theorem \ref{WellPosedIntro}.
Finally in section \ref{SectionBlowupAS}, we will outline the properties of axisymmetric, swirl-free solutions of the inviscid vortex stretching equation, and we will use these properties to prove Theorem \ref{BlowupThmIntro}, establishing finite-time blowup.

\section{Definitions, notation, and key lemmas} \label{SectionDefinitions}

We will begin by defining Hilbert spaces and recalling a key result.

\begin{definition}
For all $s\in \mathbb{R}$,
let $H^s\left(\mathbb{R}^3\right)$
be the Hilbert space with norm
\begin{equation}
    \|f\|_{H^s}
    =
    \left(\int_{\mathbb{R}^3}
    \left(1+4\pi^2|\xi|^2\right)^s
    \left|\hat{f}(\xi)\right|^2 \diff\xi
    \right)^\frac{1}{2}
\end{equation}
For all $s>-\frac{3}{2}$, we will define the homogeneous version of this space $\dot{H}^s\left(\mathbb{R}^3\right)$ by the norm
\begin{equation}
    \|f\|_{\dot{H}^s}
    =\left(\int_{\mathbb{R}^3}
    (2\pi|\xi|)^{2s} 
    \left|\hat{f}(\xi)\right|^2
    \diff\xi\right)^\frac{1}{2}.
\end{equation}
\end{definition}

\begin{lemma} \label{HilbertAlgebra}
    If $s>\frac{3}{2}$, then $H^s\left(\mathbb{R}^3\right)$ is an algebra.
    In particular, for all $f,g\in H^s\left(\mathbb{R}^3\right)$,
    \begin{equation}
        \|fg\|_{H^s} 
        \leq
        C_s \|f\|_{H^s} \|g\|_{H^s},
    \end{equation}
    where $C_s>0$ is an absolute constant depending only on $s>\frac{3}{2}$.
\end{lemma}

\begin{remark}
Throughout this paper, the constant $C_s$ will always refer to the constant in Lemma \ref{HilbertAlgebra}. This result is a classical analysis bound for Sobolev spaces. For more details, see Lemma 3.4 in Majda and Bertozzi \cite{MajdaBertozzi} and the proof in the appendix of chapter 3.
\end{remark}

\begin{definition}
For any Banach space $X\left(\mathbb{R}^3\right)$
of functions over $\mathbb{R}^3$, we will take the Banach space $C^1\left([0,T];
X\left(\mathbb{R}^3\right)\right)$,
to be the Banach space with norm
\begin{equation}
    \|f\|_{C^1\left([0,T];X\right)}
    =
    \sup_{0\leq t\leq T}\|f(\cdot,t)\|_X
    +\sup_{0\leq t\leq T}
    \|\partial_t f(\cdot,t)\|_X.
\end{equation}
\end{definition}

We will also define the space of gradients and divergence free vector fields, state the Helmholtz decomposition, as well as a related isometry for $\omega$ and $\nabla u$.

\begin{definition}
For all $s\in\mathbb{R}$, the space of divergence free vector fields $H^s_{df}\subset H^s$ is given by
\begin{equation}
    H^s_{df}\left(\mathbb{R}^3;\mathbb{R}^3\right)
    =
    \left\{v\in
    H^s\left(\mathbb{R}^3;\mathbb{R}^3\right)
    : \xi\cdot\hat{v}(\xi)=0\right\}
\end{equation}
\end{definition}

\begin{definition}
For all $s\in\mathbb{R}$, the space of gradients $H^s_{gr}\subset H^s$ is given by
\begin{equation}
    H^s_{gr}\left(\mathbb{R}^3;\mathbb{R}^3\right)
    =
    \left\{v\in
    H^s\left(\mathbb{R}^3;\mathbb{R}^3\right)
    : \xi\times\hat{v}(\xi)=0\right\}
\end{equation}
\end{definition}

\begin{lemma} \label{HelmholtzDecomp}
For all $s\in\mathbb{R}$ and for all
$u\in H^s_{df}\left(\mathbb{R}^3\right)$
and $\nabla f\in H^s_{gr}\left(\mathbb{R}^3\right)$,
\begin{equation}
    \left<u,\nabla f\right>_{H^s}=0.
\end{equation}
Furthermore, for all $s\in\mathbb{R}$ and for all $v\in H^s\left(\mathbb{R}^3;\mathbb{R}^3\right)$
there exists a unique $u\in H^s_{df}\left(\mathbb{R}^3\right)$
and $\nabla f\in H^s_{gr}\left(\mathbb{R}^3\right)$
such that
\begin{equation}
    v=u+\nabla f,
\end{equation}
and 
\begin{equation}
    \|v\|_{H^s}^2
    =
    \|u\|_{H^s}^2+\|\nabla f\|_{H^s}^2.
\end{equation}

We will define the Helmholtz projection onto divergence free vector fields,
$P_{df}: H^s \to H^s_{df}$, by
\begin{equation}
    P_{df}(v)=u,
\end{equation}
noting that this implies that for all $v\in 
H^s\left(\mathbb{R}^3;\mathbb{R}^3\right)$,
\begin{equation}
    \|P_{df}(v)\|_{H^s}\leq
    \|v\|_{H^s}.
\end{equation}

\end{lemma}

\begin{proposition} \label{VortIsometry}
    For all $s\in\mathbb{R}$, and for all
    $\mathbf{\omega}\in H^s_{df}\left(\mathbb{R}^3\right)$,
    \begin{equation}
        \|\mathbf{\omega}\|_{H^s}
        =
        \|\nabla u\|_{H^s},
    \end{equation}
    where $u=\nabla \times
    (-\Delta)^{-1}\mathbf{\omega}$
    and equivalently
    $\omega=\nabla\times u$.
\end{proposition}

\begin{remark}
This is a classical result in the mathematical fluid mechanics literature, and in some forms goes back to Helmholtz \cite{Helmholtz}. It is a two line computation on the Fourier space side.
\end{remark}

Finally, we will introduce cylindrical coordinates, and axisymmetric vector fields.

\begin{definition}
For all $x\in \mathbb{R}^3$,
\begin{align}
    r&=\sqrt{x_1^2+x_2^2} \\
    z&= x_3 \\
    e_r&=(\cos\theta,\sin\theta,0) \\
    e_\theta&= (-\sin\theta,\cos\theta,0) \\
    e_z&= e_3.
\end{align}
\end{definition}

\begin{definition}
Suppose $u\in H^s_{df}\left(\mathbb{R}^3\right),
s>\frac{3}{2}$,
then we will say that $u$ is axisymmetric and swirl-free if and only if 
for all $x\in\mathbb{R}^3$
\begin{equation}
    u(x)=u_r(r,z)e_r+u_z(r,z)e_z.
\end{equation}
\end{definition}

\begin{definition}
Suppose $\omega\in H^s_{df}\left(\mathbb{R}^3\right),
s>\frac{3}{2}$,
then we will say that $\omega \in H^s_*
\left(\mathbb{R}^3\right)$
if and only if 
for all $x\in\mathbb{R}^3$
\begin{equation}
    \omega(x)=\omega_\theta(r,z)e_\theta.
\end{equation}
\end{definition}

\begin{remark}
We define these spaces analogously for homogeneous Hilbert spaces. Note that these two spaces have the following relationship. $u\in \dot{H}^1_{df}\cap H^{s+1}_{df}$ is axisymmetric and swirl-free if and only if $\nabla\times u\in H^s_*$. For this reason, we will say that $\omega$ is an axisymmetric, swirl-free solution of the inviscid vortex stretching equation if $\omega\in H^s_*$, which means it is the vorticity associated with axisymmetric, swirl-free velocity.
These are the solutions where $u$ is axisymmetric and swirl-free.
\end{remark}

\begin{definition}
To avoid confusion with two dimensional space, we will refer to all of $\mathbb{R}^3$ in cylindrical coordinates as
\begin{equation}
    \Pi=\left\{(r,z): r\geq 0,
    z\in\mathbb{R}\right\},
\end{equation}
and the upper half space $\mathbb{R}^3_+$ in cylindrical corrdinates as
\begin{equation}
    \Pi_+=\left\{(r,z): r\geq 0,
    z\geq 0 \right\}.
\end{equation}
For $\omega\in H^s_*$ will let $\Omega \subset \Pi$ be the support of the vorticity in cylindrical coordinates,
\begin{equation}
    \Omega= \supp(\omega_\theta),
\end{equation}
and we will let $\Omega_+\subset \Pi_+$ be the support of the vorticity in the upper half space,
\begin{equation}
    \Omega_+=\Omega \cap \Pi_+.
\end{equation}
\end{definition}

\section{Local wellposedness} \label{SectionLocalWP}

We will begin by defining a bilinear operator related to the inviscid vortex stretching equation and proving a key bound on this operator.

\begin{definition}
For all $s>\frac{3}{2}$, we define the bilinear operator $B:H^s_{df}\left(\mathbb{R}^3\right)
\times H^s_{df}\left(\mathbb{R}^3\right)
\to H^s_{df}\left(\mathbb{R}^3\right)$ by
\begin{equation}
    B\left(\mathbf{\omega},
    \mathbf{\Tilde{\omega}}\right)
    =
    \frac{1}{2}P_{df}\left((\omega\cdot\nabla)
    \Tilde{u}+(\Tilde{\omega}\cdot\nabla)u\right).
\end{equation}
\end{definition}

\begin{lemma}
For all $\omega,\Tilde{\omega} \in
H^s_{df}\left(\mathbb{R}^3\right), s>\frac{3}{2}$,
\begin{equation}
    \|B(\omega,\Tilde{\omega})\|_{H^s}
    \leq 
    C_s
    \|\omega\|_{H^s}
    \|\Tilde{\omega}\|_{H^s}
\end{equation}
\end{lemma}
\begin{proof}
Applying the boundedness of the Helmholtz projection from Lemma \ref{HelmholtzDecomp}, the inequality in Lemma \ref{HilbertAlgebra}, and the isometry in Proposition \ref{VortIsometry}, we find that
\begin{align}
    \left\|P_{df}\left((\omega\cdot\nabla)
    \Tilde{u}\right)\right\|_{H^s}
    &\leq
    \left\|(\omega\cdot\nabla)
    \Tilde{u}\right\|_{H^s} \\
    &\leq 
    C_s\|\omega\|_{H^s} 
    \|\nabla \Tilde{u}\|_{H^s} \\
    &= 
    C_s \|\omega\|_{H^s} \|\Tilde{\omega}\|_{H^s}.
\end{align}
By precisely the same arguments we can see that 
\begin{equation}
    \left\|P_{df}\left((\Tilde{\omega}\cdot\nabla)
    u\right)\right\|_{H^s}
    \leq 
    C_s \|\omega\|_{H^s} \|\Tilde{\omega}\|_{H^s}.
\end{equation}
Finally, we can conlude that
\begin{align}
    \|B(\omega,\Tilde{\omega})\|_{H^s}
    &\leq 
    \frac{1}{2}\left\|P_{df}\left((\omega\cdot\nabla)
    \Tilde{u}\right)\right\|_{H^s}
    +\frac{1}{2}\left\|P_{df}\left(
    (\Tilde{\omega}\cdot\nabla)
    u\right)\right\|_{H^s} \\
    &\leq 
    C_s \|\omega\|_{H^s} \|\Tilde{\omega}\|_{H^s},
\end{align}
and this completes the proof.
\end{proof}

With this bound, we can prove the local wellposedness of the inviscid vortex stretching equation using the classic Picard iteration scheme.
We will now prove Theorem \ref{WellPosedIntro}, which will be broken into several pieces to aid the flow of the proofs.

\begin{theorem} \label{LocalWP}
For all $\omega^0\in H^s_{df}\left(\mathbb{R}^3\right), s>\frac{3}{2}$,
there exists a unique strong solution of the inviscid vortex stretching equation
$\omega\in C^1\left([0,T];
H^s_{df}\left(\mathbb{R}^3\right)\right)$,
where
\begin{equation}T<\frac{1}
{4C_s\left\|\omega^0\right\|_{H^s}},
\end{equation}
and furthermore 
\begin{equation}
    \|\omega\|_{C_T H^s_x}
    < 
    2\left\|\omega^0\right\|_{H^s}.
\end{equation}
\end{theorem}

\begin{proof}
By the fundamental theorem of calculus, we can conclude that $\omega \in C\left(\left[0,T\right];
H^s_{df}\left(\mathbb{R}^3\right)\right)$, is a strong solution of the inviscid vortex stretching equation if and only if for all $0\leq t\leq T$
\begin{equation} \label{IntegralForm}
    \omega(\cdot,t)=\omega^0+\int_0^t
    B(\omega,\omega)(\cdot,\tau) \diff\tau. 
\end{equation}
Note, in particular that if \eqref{IntegralForm} holds, then we have
\begin{equation}
    \partial_t \omega =B(\omega,\omega),
\end{equation}
and therefore the evolution equation is satisfied, and we have
$\partial_t\omega \in C\left(\left[0,T\right];
H^s_{df}\left(\mathbb{R}^3\right)\right)$.

Now we will show that there is a unique solution to the integral equation \eqref{IntegralForm} using Picard iteration.
First we will define $X \subset C\left([0,T];
H^s_{df}\right)$ by
\begin{equation}
X=\left\{\omega \in C\left([0,T];
H^s_{df}\right): 
\|\omega\|_{C_t H^s_x}\leq 2\left\|\omega^0\right\|_{H^s}
\right\},
\end{equation}
and the map $Q: X\to X$ by
\begin{equation}
    Q[w](\cdot,t)
    =
    \omega^0
    +\int_0^t B(w,w)(\cdot,\tau) \diff\tau.
\end{equation}
First we will confirm that $Q$ is in fact an automorphism on $X$.
Suppose $w \in X$.
For all $0\leq t\leq T$,
\begin{align}
    \|Q[w](\cdot,t)\|_{H^s}
    &\leq 
    \left\|\omega^0\right\|_{H^s}
    +\int_0^t\|B(\omega,\omega)
    (\cdot,\tau)\|_{H^s} \diff\tau \\
    &\leq 
    \left\|\omega^0\right\|_{H^s}
    +\int_0^t\|B(\omega,\omega)
    (\cdot,\tau)\|_{H^s} \diff\tau \\
    &\leq 
    \left\|\omega^0\right\|_{H^s}
    +C_s T \|\omega(\cdot,\tau)
    \|^2_{C_T H^s_x} \\
    & \leq
    \left\|\omega^0\right\|_{H^s}
    +4C_s T \left\|\omega^0\right\|^2_{H^s} \\
    &<
    2 \left\|\omega^0\right\|_{H^s}.
\end{align}

Next, we will show that $Q$ is a contraction mapping on $X$. Take $w,\Tilde{w}\in X$, and also let
\begin{align}
    v&= \nabla\times (-\Delta)^{-1}w \\
    \Tilde{v}&= \nabla\times 
    (-\Delta)^{-1}\Tilde{w}.
\end{align}
We can see that
\begin{equation}
    Q[w](\cdot,t)-Q[\Tilde{w}](\cdot,t)
    =\int_0^t B(w,w)(\cdot,\tau)-
    B(\Tilde{w},\Tilde{w})(\cdot,\tau)
    \diff\tau.
\end{equation}
We then compute that
\begin{align}
    B(w,w)-B(\Tilde{w},\Tilde{w})
    &=
    P_{df}((w\cdot\nabla)v-
    (\Tilde{w}\cdot\nabla)\Tilde{v}) \\
    &=
    \frac{1}{2}P_{df}((w+\Tilde{w})
    \cdot\nabla(v-\Tilde{v})
    +(w-\Tilde{w})
    \cdot\nabla(v+\Tilde{v})) \\
    &=
    B(w+\Tilde{w},w-\Tilde{w})
\end{align}
Then for all $0\leq t \leq T$,
\begin{align}
    \|Q[w](\cdot,t)-Q[\Tilde{w}](\cdot,t)\|_{H^s}
    &\leq 
    \int_0^t \|B(w+\Tilde{w},w-\Tilde{w})
    (\cdot,\tau)\|_{H^s} \diff\tau \\
    &\leq 
    C_s \int_0^t \|w(\cdot,\tau)
    +\Tilde{w}(\cdot,\tau)\|_{H^s}
    \|w(\cdot,\tau)
    -\Tilde{w}(\cdot,\tau)\|_{H^s} \diff\tau \\
    &\leq 
    4C_s \left\|\omega^0\right\|_{H^s} T
    \|w-\Tilde{w}\|_{C_T H^s_x}.
\end{align}
Observe that 
$4C_s \left\|\omega^0\right\|_{H^s} T<1$,
and so $Q$ is a contraction mapping,
and by the Picard contraction mapping theorem, there exists a unique $\omega\in X$ such that
\begin{equation}
    \omega=Q[\omega].
\end{equation}
We have now shown there exists a unique 
$\omega \in C_T H^s_x$ such that 
for all $0\leq t\leq T$,
\begin{equation}
    \omega(\cdot,t)
    =
    \omega^0
    +\int_0^t B(\omega,\omega)(\cdot,t)
    \diff\tau.
\end{equation}
We have shown that this is an equivalent condition to $\omega$ being a strong solution of the inviscid vortex stretching equation, and so this completes the proof.
\end{proof}

\begin{corollary} \label{BlowupTimeBound}
For all $\omega^0\in H^s_{df}\left(\mathbb{R}^3\right), s>\frac{3}{2}$,
there exists a unique strong solution of the inviscid vortex stretching equation
$\omega\in C^1\left(\left[0,T_{max}\right);
H^s_{df}\left(\mathbb{R}^3\right)\right)$, where \begin{equation}
T_{max}
\geq 
\frac{1}{2C_s\left\|\omega^0\right\|_{H^s}}.
\end{equation}
\end{corollary}

\begin{proof}
We begin by fixing $0<\epsilon<1$, which can be taken to be arbitrarily small, and by letting
\begin{equation}
    T_0=\frac{1-\epsilon}{4C_s
    \left\|\omega^0\right\|_{H^s}}
\end{equation}
Applying Theorem \ref{LocalWP}, we can see that there exists a unique strong solution of the inviscid vortex stretching equation
$\omega\in C^1\left(\left[0,T_0\right];
H^s_{df}\left(\mathbb{R}^3\right)\right)$
and furthermore that
\begin{equation}
    \|\omega\|_{C_{T_0} H^s_x}
    <
    2\left\|\omega^0\right\|_{H^s}.
\end{equation}

Now we will take
\begin{equation}
    \omega^1=\omega\left(\cdot,T_0\right),
\end{equation}
observing that
\begin{equation}
    \left\|\omega^1\right\|_{H^s}
    < 
    2\left\|\omega^0\right\|_{H^s}.
\end{equation}
Again applying Theorem \ref{LocalWP}, we can conclude that there exists a unique strong solution of the inviscid vortex stretching equation
$\omega\in C^1\left(\left[0,T_0+T_1\right];
H^s_{df}\left(\mathbb{R}^3\right)\right)$,
where 
\begin{align}
    T_1
    &=
    \frac{1-\epsilon}{4C_s 
    \left\|\omega^1\right\|_{H^s}} \\
    &>
    \frac{1-\epsilon}{8C_s 
    \left\|\omega^0\right\|_{H^s}}.
\end{align}
and for all $0\leq t \leq T_0+T_1$,
\begin{align}
    \|\omega(\cdot,t)\|_{H^s}
    &<
    2 \left\|\omega^1\right\|_{H^s} \\
    &<
    4 \left\|\omega^0\right\|_{H^s}.
\end{align}

Our iteration scheme will continue in this way.
For all $k\in\mathbb{N}$, we will let
\begin{equation}
    \omega^k=\omega\left(\cdot,
    \sum_{j=0}^k T_j\right),
\end{equation}
and
\begin{equation}
    T_{k+1}=\frac{1-\epsilon}{4C_s
    \left\|\omega^k\right\|_{H^s}},
\end{equation}
Applying Theorem \ref{LocalWP}, we can see that
\begin{equation}
    \left\|\omega^{k+1}\right\|_{H^s}
    <
    2 \left\|\omega^k\right\|_{H^s}.
\end{equation}
By induction, for all $k\in \mathbb{N}$,
\begin{equation}
    \left\|\omega^k\right\|_{H^s}
    <
    2^k
    \left\|\omega^0\right\|_{H^s},
\end{equation}
and therefore
\begin{equation}
    T_{k} >\frac{1-\epsilon}{4C_s
    \left\|\omega^0\right\|_{H^s}} 
    \left(\frac{1}{2}\right)^k.
\end{equation}
Clearly by Theorem \ref{LocalWP},
\begin{align}
    T_{max}
    &\geq 
    \sum_{k=0}^{+\infty} T_k \\
    &>
    \frac{1-\epsilon}{4C_s
    \left\|\omega^0\right\|_{H^s}} 
    \sum_{k=0}^{+\infty} 
    \left(\frac{1}{2}\right)^k \\
    &= \label{EpsilonStep}
    \frac{1-\epsilon}{2C_s
    \left\|\omega^0\right\|_{H^s}}
\end{align}
Recall that $0<\epsilon<1$ is arbitrary, and so \eqref{EpsilonStep} holds for all $0<\epsilon<1$.
Taking the limit $\epsilon \to 0$, we can conclude that
\begin{equation}
    T_{max}
    \geq 
    \frac{1}{2C_s
    \left\|\omega^0\right\|_{H^s}}.
\end{equation}
This completes the proof.
\end{proof}

\begin{corollary} \label{HilbertGrowthLowerBound}
Suppose $\omega\in C^1\left(\left[0,T_{max}\right);
H^s_{df}\left(\mathbb{R}^3\right)\right), s>\frac{3}{2}$ is a solution to the inviscid vortex stretching equation.
If $T_{max}<+\infty,$ then for all $0\leq t<T_{max}$,
\begin{equation}
    \|\omega(\cdot,t)\|_{H^s}
    \geq
    \frac{1}{2C_s (T_{max}-t)}.
\end{equation}
\end{corollary}

\begin{proof}
We will treat $\omega(\cdot,t)$ as initial data. Applying Corollary \ref{BlowupTimeBound}, we find that
\begin{equation}
    T_{max}-t
    \geq
    \frac{1}{2C_s\left\|\omega(\cdot,t)
    \right\|_{H^s}},
\end{equation}
and this completes the proof.
\end{proof}

\section{Axisymmetric, swirl-free solutions}
\label{SectionBlowupAS}

In this section, we will consider axisymmetric, swirl-free solutions of the inviscid vortex stretching equation, and will show that this equation exhibits finite-time blowup for a subset of these solutions with colliding vortex ring geometry.
We begin by giving a Biot-Savart law for axisymmetric swirl-free vector fields.

\begin{proposition} \label{BiotSavart}
    For all $u \in \dot{H}^1_{df}\cap \dot{H}^{s+1}_{df}, s> \frac{3}{2}, u$ is axisymmetric and swirl-free if and only if
    \begin{equation}
        \omega(x)=\omega_{\theta}(r,z)e_\theta.
    \end{equation}
    Furthermore, if $u$ is axisymmetric and swirl-free, then
    \begin{align}
        u_r(r,z) &=
        \frac{1}{4\pi} \int_{-\infty}^{+\infty}
        \int_0^{+\infty} 
        \int_{-\pi}^\pi
        \frac{\Bar{r}(z-\Bar{z})\cos(\phi)}
        {\left((z-\Bar{z})^2+r^2+\Bar{r}^2
        -2r\Bar{r}\cos(\phi)\right)^\frac{3}{2}}
        \omega_{\theta}(\Bar{r},\Bar{z})
        \diff\phi \diff\Bar{r} \diff\Bar{z}\\
        u_z(r,z) &= 
        \frac{1}{4\pi} 
        \int_{-\infty}^{\infty}
        \int_0^\infty 
        \int_{-\pi}^\pi
        \frac{\bar{r}\left(z-\bar{z}\right)
        (\bar{r}-r\cos(\phi))}
        {\left((z-\bar{z})^2+r^2+\bar{r}^2
        -2r\bar{r} \cos(\phi)
        \right)^\frac{3}{2}} 
        \omega_\theta(\bar{r},\bar{z})
        \diff\phi \diff\bar{r} \diff\bar{z}
    \end{align}
\end{proposition}

\begin{proof}
For the forward direction, this is a standard vector calculus identity. For the reverse direction, it suffices to show that when $\omega(x)=\omega_{\theta}(r,z) e_{\theta}$,
then for all $x\in\mathbb{R}^3$,
\begin{align}
    e_r \cdot u(x)&= 
        \frac{1}{4\pi} \int_{-\infty}^{+\infty}
        \int_0^{+\infty} 
        \int_{-\pi}^\pi
        \frac{\Bar{r}(z-\Bar{z})\cos(\phi)}
        {\left((z-\Bar{z})^2+r^2+\Bar{r}^2
        -2r\Bar{r}\cos(\phi)\right)^\frac{3}{2}}
        \omega_{\theta}(\Bar{r},\Bar{z})
        \diff\phi \diff\Bar{r} \diff\Bar{z} \\
    e_\theta \cdot u(x) &=0 \\
    e_3 \cdot u(x) &= 
        \frac{1}{4\pi} 
    \int_{-\infty}^{\infty}
    \int_0^\infty 
    \int_{-\pi}^\pi
    \frac{\bar{r}\left(z-\bar{z}\right)
    (\bar{r}-r\cos(\phi))}
    {\left((z-\bar{z})^2+r^2+\bar{r}^2
    -2r\bar{r} \cos(\phi)
    \right)^\frac{3}{2}} 
    \omega_\theta(\bar{r},\bar{z})
    \diff\phi \diff\bar{r} \diff\bar{z}.
\end{align}

We will begin by recalling that 
\begin{equation}
    u=\nabla \times (-\Delta)^{-1}\omega.
\end{equation}
We know that 
\begin{equation}
    (-\Delta)^{-1}\omega(x)
    =
    \frac{1}{4 \pi}\int_{\mathbb{R}^3}
    \frac{1}{|x-y|}\omega(y)\diff y,
\end{equation}
and so we can see that
\begin{align}
    u(x)
    &=
    \frac{1}{4 \pi}\int_{\mathbb{R}^3}
    \nabla\left(\frac{1}{|x-y|}\right)
    \times\omega(y)\diff y \\
    &=
    -\frac{1}{4 \pi}\int_{\mathbb{R}^3}
    \frac{x-y}{|x-y|^3} \times\omega(y)\diff y.
\end{align}
Moving to cylindrical coordinates we have
\begin{align}
    x&=(r\cos(\theta),r\sin(\theta),z) \\
    y&=(\bar{r}\cos(\phi),
    \bar{r}\sin(\phi),\bar{z}),
\end{align}
and so we can compute that
\begin{align}
    |x-y|^2
    &=
    (r\cos(\theta)-\bar{r}\cos(\phi))^2
    +(r\sin(\theta)-\bar{r}\sin(\phi))^2
    +(z-\bar{z})^2 \\
    &=
    (z-\bar{z})^2+r^2+\bar{r}^2
    -2r\bar{r}(\cos(\theta)\cos(\phi)
    +\sin(\theta)\sin(\phi)) \\
    &=
    (z-\bar{z})^2+r^2+\bar{r}^2
    -2r\bar{r} \cos(\phi-\theta).
\end{align}
Computing the radial component and switching to cylindrical coordinates, we have
\begin{align}
    e_r \cdot u(x)
    &=
    -\frac{1}{4\pi} 
    e_r \cdot
    \int_{-\infty}^{\infty}
    \int_0^\infty 
    \int_{-\pi}^\pi
    \bar{r} \frac{x-y}{|x-y|^3} \times 
    \omega_\theta(\bar{r},\bar{z}) e_\phi
    \diff \phi \diff\bar{r} \diff\bar{z} \\
    &=
    \frac{1}{4\pi} 
    \int_{-\infty}^{\infty}
    \int_0^\infty 
    \int_{-\pi}^\pi
    \bar{r} \omega_\theta(\bar{r},\bar{z})
    \frac{x-y}{|x-y|^3} \cdot  
    \left(\begin{array}{c}
          \cos(\theta)  \\
          \sin(\theta) \\
          0
    \end{array}\right)
    \times 
    \left(\begin{array}{c}
          -\sin(\phi)  \\
          \cos(\phi) \\
          0
    \end{array}\right)
    \diff \phi \diff\bar{r} \diff\bar{z} \\
    &=
    \frac{1}{4\pi} 
    \int_{-\infty}^{\infty}
    \int_0^\infty 
    \int_{-\pi}^\pi
    \frac{\bar{r}\left(z-\bar{z}\right)
    \omega_\theta(\bar{r},\bar{z})}
    {\left((z-\bar{z})^2+r^2+\bar{r}^2
    -2r\bar{r} \cos(\phi-\theta)
    \right)^\frac{3}{2}} 
    \cos(\phi-\theta)
    \diff \phi \diff\bar{r} \diff\bar{z} \\
    &=
    \frac{1}{4\pi} 
    \int_{-\infty}^{\infty}
    \int_0^\infty 
    \int_{-\pi}^\pi
    \frac{\bar{r}\left(z-\bar{z}\right)}
    {\left((z-\bar{z})^2+r^2+\bar{r}^2
    -2r\bar{r} \cos(\phi)
    \right)^\frac{3}{2}} 
    \omega_\theta(\bar{r},\bar{z})
    \cos(\phi)
    \diff \phi \diff\bar{r} \diff\bar{z}.
\end{align}
Note that in the last step we used that the integral is independent of $\theta$ because we are integrating over the whole period, and so a shift does not matter.

Likewise, we will compute that
\begin{align}
    e_\theta \cdot u(x)
    &=
    \frac{1}{4\pi} 
    \int_{-\infty}^{\infty}
    \int_0^\infty 
    \int_{-\pi}^\pi
    \bar{r} \omega_\theta(\bar{r},\bar{z})
    \frac{x-y}{|x-y|^3} \cdot  
    \left(\begin{array}{c}
          -\sin(\theta)  \\
          \cos(\theta) \\
          0
    \end{array}\right)
    \times 
    \left(\begin{array}{c}
          -\sin(\phi)  \\
          \cos(\phi) \\
          0
    \end{array}\right)
    \diff \phi \diff\bar{r} \diff\bar{z} \\
    &= 
    \frac{1}{4\pi} 
    \int_{-\infty}^{\infty}
    \int_0^\infty 
    \int_{-\pi}^\pi
    \frac{\bar{r}\left(z-\bar{z}\right)
    \omega_\theta(\bar{r},\bar{z})}
    {\left((z-\bar{z})^2+r^2+\bar{r}^2
    -2r\bar{r} \cos(\phi-\theta)
    \right)^\frac{3}{2}} 
    \sin(\phi-\theta)
    \diff \phi \diff\bar{r} \diff\bar{z} \\
    &=
    \frac{1}{4\pi} 
    \int_{-\infty}^{\infty}
    \int_0^\infty 
    \int_{-\pi}^\pi
    \frac{\bar{r}\left(z-\bar{z}\right)}
    {\left((z-\bar{z})^2+r^2+\bar{r}^2
    -2r\bar{r} \cos(\phi)
    \right)^\frac{3}{2}} 
    \omega_\theta(\bar{r},\bar{z})
    \sin(\phi)
    \diff \phi \diff\bar{r} \diff\bar{z} \\
    &=
    0.
\end{align}
Finally we will compute,
\begin{align}
    e_3 \cdot u(x)
    &=
    -\frac{1}{4\pi} 
    e_3 \cdot
    \int_{-\infty}^{\infty}
    \int_0^\infty 
    \int_{-\pi}^\pi
    \bar{r} \frac{x-y}{|x-y|^3} \times 
    \omega_\theta(\bar{r},\bar{z}) e_\phi
    \diff \phi \diff\bar{r} \diff\bar{z} \\
    &=
    \frac{1}{4\pi} 
    \int_{-\infty}^{\infty}
    \int_0^\infty 
    \int_{-\pi}^\pi
    \bar{r} \omega_\theta(\bar{r},\bar{z})
    \frac{x-y}{|x-y|^3} \cdot 
    \left(e_3 \times e_\phi \right)
    \diff \phi \diff\bar{r} \diff\bar{z} \\
    &=
    \frac{1}{4\pi} 
    \int_{-\infty}^{\infty}
    \int_0^\infty 
    \int_{-\pi}^\pi
    \bar{r} \omega_\theta(\bar{r},\bar{z})
    \frac{x-y}{|x-y|^3} \cdot 
    \left(\begin{array}{c}
         -\cos(\phi)  \\
         -\sin(\phi) \\
         0
    \end{array}\right)
    \diff \phi \diff\bar{r} \diff\bar{z}.
    \end{align}
We can compute that
\begin{align}
    (x-y)\cdot 
    \left(\begin{array}{c}
         -\cos(\phi)  \\
         -\sin(\phi) \\
         0
    \end{array}\right)  
    &=
    \cos(\phi)(\bar{r}\cos(\phi)
    -r\cos(\theta))
    +\sin(\phi)(\bar{r}\sin(\phi)
    -r\sin(\theta)) \\
    &=
    \bar{r}-r\cos(\phi-\theta).
\end{align}
Plugging this in we find that
\begin{align}
    e_3 \cdot u(x)
    &=
    \frac{1}{4\pi} 
    \int_{-\infty}^{\infty}
    \int_0^\infty 
    \int_{-\pi}^\pi
    \frac{\bar{r}\left(z-\bar{z}\right)
    (\bar{r}-r\cos(\phi-\theta))
    \omega_\theta(\bar{r},\bar{z})}
    {\left((z-\bar{z})^2+r^2+\bar{r}^2
    -2r\bar{r} \cos(\phi-\theta)
    \right)^\frac{3}{2}} 
    \diff\phi \diff\bar{r} \diff\bar{z} \\
    &=
    \frac{1}{4\pi} 
    \int_{-\infty}^{\infty}
    \int_0^\infty 
    \int_{-\pi}^\pi
    \frac{\bar{r}\left(z-\bar{z}\right)
    (\bar{r}-r\cos(\phi))}
    {\left((z-\bar{z})^2+r^2+\bar{r}^2
    -2r\bar{r} \cos(\phi)
    \right)^\frac{3}{2}} 
    \omega_\theta(\bar{r},\bar{z})
    \diff\phi \diff\bar{r} \diff\bar{z},
\end{align}
and this completes the proof.
\end{proof}

\begin{proposition} \label{BiotSavartOdd}
    Suppose $\omega \in H^s_*, s>\frac{3}{2}$,
    is a vorticity vector corresponding to an axisymetric, swirl-free vector field,
    and suppose $\omega_{\theta}(r,z)$ is odd in $z$.
    Then we have
\begin{multline}
    u_r(r,z)
    =
    \frac{1}{2\pi} \int_0^\infty \int_0^\infty
    \int_0^1
    \Bigg(
    \frac{\Bar{r}(z-\Bar{z})
    s\left(1-s^2\right)^{-\frac{1}{2}}}{\left((z-\Bar{z})^2+r^2+\Bar{r}^2
    -2r\Bar{r}s\right)^\frac{3}{2}}
    -\frac{\Bar{r}(z-\Bar{z})
    s\left(1-s^2\right)^{-\frac{1}{2}}}{\left((z-\Bar{z})^2+r^2+\Bar{r}^2
    +2r\Bar{r}s\right)^\frac{3}{2}}
    \\
    -\frac{\Bar{r}(z+\Bar{z})
    s\left(1-s^2\right)^{-\frac{1}{2}}}{\left((z+\Bar{z})^2+r^2+\Bar{r}^2
    -2r\Bar{r}s\right)^\frac{3}{2}}
    +\frac{\Bar{r}(z+\Bar{z})
    s\left(1-s^2\right)^{-\frac{1}{2}}}{\left((z+\Bar{z})^2+r^2+\Bar{r}^2
    +2r\Bar{r}s\right)^\frac{3}{2}}
    \Bigg) \omega_\theta(\bar{r},\bar{z})
    \diff s \diff\bar{r} \diff\bar{z},
\end{multline}
    noting in particular that $u_r(r,z)$ is even in $z$.
\end{proposition}

\begin{proof}
Recalling the identity for $u_r$ from Proposition \ref{BiotSavart}, we find that
\begin{align}
    u_r(r,z) 
    &=
        \frac{1}{4\pi} \int_{-\infty}^{\infty}
        \int_0^{\infty} 
        \int_{-\pi}^\pi
        \frac{\Bar{r}(z-\Bar{z})\cos(\phi)}
        {\left((z-\Bar{z})^2+r^2+\Bar{r}^2
        -2r\Bar{r}\cos(\phi)\right)^\frac{3}{2}}
        \omega_{\theta}(\Bar{r},\Bar{z})
        \diff\phi \diff\Bar{r} \diff\Bar{z}\\
    &=
        \frac{1}{2\pi} \int_{-\infty}^{\infty}
        \int_0^{\infty} 
        \int_0^\pi
        \frac{\Bar{r}(z-\Bar{z})\cos(\phi)}
        {\left((z-\Bar{z})^2+r^2+\Bar{r}^2
        -2r\Bar{r}\cos(\phi)\right)^\frac{3}{2}}
        \omega_{\theta}(\Bar{r},\Bar{z})
        \diff\phi \diff\Bar{r} \diff\Bar{z},
\end{align}
because the integrand is even in $\phi$.
Making the substitution
\begin{equation}
    s=\cos(\phi),
\end{equation}
we can see that
\begin{align}
    \diff s &
    =
    -\sin(\phi)\diff \phi \\
    &=
    -(1-s^2)^\frac{1}{2} \diff\phi.
\end{align}
Therefore, we may compute that
\begin{equation}
    u_r(r,z)
    =
    \frac{1}{2\pi} \int_{-\infty}^{\infty}
        \int_0^{\infty} 
        \int_{-1}^1
        \frac{\Bar{r}(z-\Bar{z})
        s\left(1-s^2\right)^{-\frac{1}{2}}}
        {\left((z-\Bar{z})^2+r^2+\Bar{r}^2
        -2r\Bar{r}s\right)^\frac{3}{2}}
        \omega_{\theta}(\Bar{r},\Bar{z})
        \diff s \diff\Bar{r} \diff\Bar{z}.
\end{equation}
Note that
\begin{equation}
    \int_{-1}^0
    \frac{s\left(1-s^2\right)^{-\frac{1}{2}}}
    {\left((z-\Bar{z})^2+r^2+\Bar{r}^2
    -2r\Bar{r}s\right)^\frac{3}{2}}
    \diff s
    =
    -\int_0^1
    \frac{s\left(1-s^2\right)^{-\frac{1}{2}}}
    {\left((z-\Bar{z})^2+r^2+\Bar{r}^2
    +2r\Bar{r}s\right)^\frac{3}{2}}
    \diff s.
\end{equation}
Plugging this in, we may conclude that
\begin{multline}
    u_r(r,z)
    =
    \frac{1}{2\pi} \int_{-\infty}^{\infty}
    \int_0^{\infty} 
    \int_0^1
    \left(\frac{1}{\left((z-\Bar{z})^2+r^2+\Bar{r}^2
    -2r\Bar{r}s\right)^\frac{3}{2}}
    -\frac{1}{\left((z-\Bar{z})^2+r^2+\Bar{r}^2
    +2r\Bar{r}s\right)^\frac{3}{2}}\right) \\
    \Bar{r}(z-\Bar{z})
    s\left(1-s^2\right)^{-\frac{1}{2}}
    \omega_{\theta}(\Bar{r},\Bar{z})
    \diff s \diff\Bar{r} \diff\Bar{z}.
\end{multline}
Now applying the fact that $\omega_\theta$ is odd in $z$, we can see that
\begin{multline}
    \int_{-\infty}^{0}
    \int_0^{\infty} 
    \int_0^1
    \left(\frac{1}{\left((z-\Bar{z})^2+r^2+\Bar{r}^2
    -2r\Bar{r}s\right)^\frac{3}{2}}
    -\frac{1}{\left((z-\Bar{z})^2+r^2+\Bar{r}^2
    +2r\Bar{r}s\right)^\frac{3}{2}}\right) \\
    \Bar{r}(z-\Bar{z})
    s\left(1-s^2\right)^{-\frac{1}{2}}
    \omega_{\theta}(\Bar{r},\Bar{z})
    \diff s \diff\Bar{r} \diff\Bar{z} \\
    = -\int_0^\infty
    \int_0^\infty \int_0^1
    \left(\frac{1}{\left((z+\Bar{z})^2+r^2+\Bar{r}^2
    -2r\Bar{r}s\right)^\frac{3}{2}}
    -\frac{1}{\left((z+\Bar{z})^2+r^2+\Bar{r}^2
    +2r\Bar{r}s\right)^\frac{3}{2}}\right) \\
    \Bar{r}(z+\Bar{z})
    s\left(1-s^2\right)^{-\frac{1}{2}}
    \omega_{\theta}(\Bar{r},\Bar{z})
    \diff s \diff\Bar{r} \diff\Bar{z}.
\end{multline}
Plugging this back in we can conclude that
\begin{multline} \label{RadialOddID}
    u_r(r,z)
    =
    \frac{1}{2\pi} \int_0^\infty \int_0^\infty
    \int_0^1
    \Bigg(
    \frac{\Bar{r}(z-\Bar{z})
    s\left(1-s^2\right)^{-\frac{1}{2}}}{\left((z-\Bar{z})^2+r^2+\Bar{r}^2
    -2r\Bar{r}s\right)^\frac{3}{2}}
    -\frac{\Bar{r}(z-\Bar{z})
    s\left(1-s^2\right)^{-\frac{1}{2}}}{\left((z-\Bar{z})^2+r^2+\Bar{r}^2
    +2r\Bar{r}s\right)^\frac{3}{2}}
    \\
    -\frac{\Bar{r}(z+\Bar{z})
    s\left(1-s^2\right)^{-\frac{1}{2}}}{\left((z+\Bar{z})^2+r^2+\Bar{r}^2
    -2r\Bar{r}s\right)^\frac{3}{2}}
    +\frac{\Bar{r}(z+\Bar{z})
    s\left(1-s^2\right)^{-\frac{1}{2}}}{\left((z+\Bar{z})^2+r^2+\Bar{r}^2
    +2r\Bar{r}s\right)^\frac{3}{2}}
    \Bigg) \omega_\theta(\bar{r},\bar{z})
    \diff s \diff\bar{r} \diff\bar{z}.
\end{multline}
It is a simple computation to show from the identity \eqref{RadialOddID}, that
\begin{equation}
    u_r(r,z)=u_r(r,-z),
\end{equation}
and this completes the proof.
\end{proof}

Now we will now derive a scalar evolution equation for axisymmetric, swirl-free solutions of the inviscid vortex stretching equation, and show that this class of solutions is preserved by the dynamics of the equation.

\begin{lemma} \label{AxisymLemma}
    Suppose $\omega \in H^s_*, s>\frac{3}{2}$ and that $u=\nabla\times(-\Delta)^{-1}\omega$.
    Then $P_{df}\left((\omega\cdot\nabla)u\right)
    \in H^s_*$ with
    \begin{equation}
        P_{df}\left((\omega\cdot\nabla)u\right)(x)
        =\frac{u_r(r,z)}{r}
        \omega_\theta(r,z).
    \end{equation}
\end{lemma}

\begin{proof}
We know that
\begin{align}
    \omega(x)&=\omega_\theta(r,z)e_\theta \\
    u(x)&= u_r(r,z)e_r +u_z(r,z)e_z.
\end{align}
This implies that
\begin{equation}
    \omega\cdot\nabla
    =
    \omega_\theta \frac{1}{r}\partial_\theta.
\end{equation}
Applying the identity
\begin{equation}
    \partial_\theta e_r=e_\theta,
\end{equation}
we can see that
\begin{equation}
    (\omega\cdot\nabla)u=
    \omega_\theta \frac{u_r}{r} e_\theta.
\end{equation}
This vector field is already divergence free, so this completes the proof.
\end{proof}

\begin{theorem} \label{AxisymThm}
Suppose $\omega^0\in H^s_*, s>\frac{3}{2}$ is the vorticity for an axisymmetric, swirl free vector field. Then the solution of the inviscid vortex stretching equation for this initial data
$\omega \in C\left([0,T_{max});H^s_*\right)$
is also the vorticity for an axisymmetric, swirl-free vector field. Furthermore, for all $0<t<T_{max}$, the vorticity satisfies the evolution equation
\begin{equation}
    \partial_t \omega_\theta
    =
    \frac{u_r}{r}\omega_{\theta}.
\end{equation}
Therefore, for all $0<t<T_{max}$,
and for all $(r,z)\in \Pi$,
\begin{equation} \label{ExpID}
    \omega_\theta(r,z,t)
    =
    \omega^0_\theta(r,z)
    \exp\left(\int_0^t\frac{u_r(r,z,\tau)}{r}
    \diff\tau\right).
\end{equation}
Note this implies the support of $\omega_\theta$ is preserved by the dynamics of the equation,
\begin{equation}
    \Omega(t)=\Omega^0.
\end{equation}
\end{theorem}

\begin{proof}
Applying Lemma \ref{AxisymLemma}, if $\omega(\cdot,t)\in H^s_*$, then $\partial_t\omega(\cdot,t)\in H^s_*$ with
\begin{equation}
    \partial_t\omega=
    \frac{u_r}{r}\omega_\theta e_\theta.
\end{equation}
This implies that the class of axisymmetric, swirl-free flows are preserved by the inviscid vortex stretching equation. Note that this is not only formal because each of the Picard iterates in Theorem \ref{LocalWP} must also be in $C_TH^s_*$, and so the solution, which is the limit of these iterates, must also be in $C_T H^s_*$.
It then immediately follows that
for all $0\leq t<T_{max}$,
\begin{equation}
    \partial_t\omega_\theta
    =
    \frac{u_r}{r}\omega_\theta.
\end{equation}
Integrating this differential equation immediately yields the formula \eqref{ExpID},
and this completes the proof.
\end{proof}

\begin{remark}
This result shows that axisymmetric, swirl-free initial data leads to axisymmetric, swirl-free solutions for the inviscid vortex stretching equation just as it does for the Euler equation.
For axisymmetric, swirl free solutions of the Euler equation, the dynamics can be expressed entirely in terms of $\omega_\theta(r,z,t)$, which has an evolution equation is given by
\begin{equation}
    \partial_t\omega_\theta
    +(u\cdot\nabla)\omega_{\theta}
    -\frac{u_r}{r} \omega_\theta=0.
\end{equation}
There is also a ``divergence form'' expression of this equation of the form
\begin{equation}
    \partial_t\omega_\theta
    +\partial_r(u_r\omega_\theta)
    +\partial_z(u_z\omega_\theta)
    =0.
\end{equation}
The quantity $\frac{\omega_\theta}{r}$ is transported by the flow, with
\begin{equation}
    (\partial_t+u\cdot\nabla)
    \frac{\omega_\theta}{r}
    =0.
\end{equation}
It is this fact that leads both to global regularity when the vorticity is smooth \cite{Yudovich} and to finite-time blowup for certain nonsmooth, H\"older continuous vorticities \cite{ElgindiBlowup}.
\end{remark}

Before proving finite-time blowup, it is necessary to prove some properties of axisymmetric, swirl free solutions of the Euler equation with a particular geometry, that of colliding vortex rings.

\begin{lemma} \label{OddLemma}
Suppose $\omega\in C^1\left([0,T_{max});H^s_*\right), s>\frac{3}{2}$ is an axisymmetric, swirl free solution of the inviscid vortex stretching equation with initial data $\omega^0\in H^s_*$ satisfying $\omega^0_{\theta}(r,z)$ is odd in $z$, and for all $r,z \geq 0$,
\begin{equation}
    \omega_\theta(r,z)\leq 0.
\end{equation}
Then for all $0<t<T_{max}$,
$\omega_\theta(r,z,t)$ is odd in $z$
and for all $0<t<T_{max}, r,z \geq 0$,
\begin{equation} \label{SignCondition}
    \omega_\theta(r,z,t)\leq 0.
\end{equation}
\end{lemma}

\begin{proof}
The sign condition in \eqref{SignCondition} follows immediately from the identity \eqref{ExpID}.
Furthermore, we know from Proposition \ref{BiotSavartOdd}, that if $\omega_\theta(\cdot,t)$ is odd in $z$, then
$\frac{u_r}{r}(\cdot,t)$ is even in $z$.
This implies that $\frac{u_r}{r}\omega_\theta(\cdot,t)$ 
is odd in $z$, and so 
$\partial_t\omega_\theta(\cdot,t)$
is odd in $z$.
This implies that oddness in $z$ is preserved by the dynamics, because each of the Picard iterates in Theorem \ref{LocalWP} will be odd in $z$, and so the solution, which is the limit of the Picard iterates, must be as well.
This completes the proof.
\end{proof}

\begin{proposition} \label{DerivativeProp}
    Suppose $\omega \in C\left([0,T_{max};H^s_*\right),
    s>\frac{3}{2}$
    is a solution of the inviscid vortex stretching equation and that $\omega_\theta(r,z,t)$ is odd in $z$.
    Then for all $0<t<T_{max}$
    \begin{multline} \label{GrowthA}
        \frac{\diff}{\diff t}
        \left(-\int_0^\infty \int_0^\infty
        r^2 \omega_\theta(r,z,t)\diff r \diff z\right)
        = 
    \frac{1}{2\pi} 
    \int_0^\infty \int_0^\infty
    \int_0^\infty \int_0^\infty 
    \int_0^1
    r\bar{r}\omega_\theta(r,z,t)
    \omega_\theta(\bar{r},\bar{z},t)
    \\
    \left(\frac{(z+\Bar{z})
    s\left(1-s^2\right)^{-\frac{1}{2}}}
    {\left((z+\Bar{z})^2+r^2+\Bar{r}^2
    -2r\Bar{r}s\right)^\frac{3}{2}}
    -\frac{(z+\Bar{z})
    s\left(1-s^2\right)^{-\frac{1}{2}}}
    {\left((z+\Bar{z})^2+r^2+\Bar{r}^2
    +2r\Bar{r}s\right)^\frac{3}{2}}
    \right)
    \diff s \diff\bar{r} \diff\bar{z} 
    \diff r \diff z.
    \end{multline}
\end{proposition}

\begin{proof}
Applying Lemma \ref{AxisymLemma}, we can see that
\begin{equation}
    \frac{\diff}{\diff t}
    \int_0^\infty \int_0^\infty
    r^2 \omega_\theta(r,z,t)\diff r \diff z
    = 
    \frac{\diff}{\diff t}
    \int_0^\infty \int_0^\infty
    r u_r(r,z,t) \omega_\theta(r,z,t)\diff r \diff z.
\end{equation}
Plugging into the our identity for $u_r$ from Proposition \ref{BiotSavartOdd}, we find that
\begin{multline} \label{StepA}
    \frac{\diff}{\diff t}
    \int_0^\infty \int_0^\infty
    r^2 \omega_\theta(r,z,t)\diff r \diff z
    =
    \frac{1}{2\pi} 
    \int_0^\infty \int_0^\infty
    \int_0^\infty \int_0^\infty 
    \int_0^1
    r\bar{r}\omega_\theta(r,z,t)
    \omega_\theta(\bar{r},\bar{z},t)
    \\
    \Bigg(
    \frac{(z-\Bar{z})
    s\left(1-s^2\right)^{-\frac{1}{2}}}
    {\left((z-\Bar{z})^2+r^2+\Bar{r}^2
    -2r\Bar{r}s\right)^\frac{3}{2}}
    -\frac{(z-\Bar{z})
    s\left(1-s^2\right)^{-\frac{1}{2}}}
    {\left((z-\Bar{z})^2+r^2+\Bar{r}^2
    +2r\Bar{r}s\right)^\frac{3}{2}}
    \\
    -\frac{(z+\Bar{z})
    s\left(1-s^2\right)^{-\frac{1}{2}}}
    {\left((z+\Bar{z})^2+r^2+\Bar{r}^2
    -2r\Bar{r}s\right)^\frac{3}{2}}
    +\frac{(z+\Bar{z})
    s\left(1-s^2\right)^{-\frac{1}{2}}}
    {\left((z+\Bar{z})^2+r^2+\Bar{r}^2
    +2r\Bar{r}s\right)^\frac{3}{2}}
    \Bigg)
    \diff s \diff\bar{r} \diff\bar{z} 
    \diff r \diff z.
\end{multline}
Observe that the measure
$\bar{r}r\omega_\theta(r,z,t)
\omega_\theta(\bar{r},\bar{z},t)
\diff\bar{r}\diff\bar{z} \diff r \diff z$
is invariant with respect to the interchanging of $(r,z)$ and $\bar{r},\bar{z}$, and so
\begin{multline} \label{StepB}
    \frac{\diff}{\diff t}
    \int_0^\infty \int_0^\infty
    r^2 \omega_\theta(r,z,t)\diff r \diff z
    =
    \frac{1}{2\pi} 
    \int_0^\infty \int_0^\infty
    \int_0^\infty \int_0^\infty 
    \int_0^1
    r\bar{r}\omega_\theta(r,z,t)
    \omega_\theta(\bar{r},\bar{z},t)
    \\
    \Bigg(
    \frac{(\Bar{z}-z)
    s\left(1-s^2\right)^{-\frac{1}{2}}}
    {\left((z-\Bar{z})^2+r^2+\Bar{r}^2
    -2r\Bar{r}s\right)^\frac{3}{2}}
    -\frac{(\Bar{z}-z)
    s\left(1-s^2\right)^{-\frac{1}{2}}}
    {\left((z-\Bar{z})^2+r^2+\Bar{r}^2
    +2r\Bar{r}s\right)^\frac{3}{2}}
    \\
    -\frac{(z+\Bar{z})
    s\left(1-s^2\right)^{-\frac{1}{2}}}
    {\left((z+\Bar{z})^2+r^2+\Bar{r}^2
    -2r\Bar{r}s\right)^\frac{3}{2}}
    +\frac{(z+\Bar{z})
    s\left(1-s^2\right)^{-\frac{1}{2}}}
    {\left((z+\Bar{z})^2+r^2+\Bar{r}^2
    +2r\Bar{r}s\right)^\frac{3}{2}}
    \Bigg)
    \diff s \diff\bar{r} \diff\bar{z} 
    \diff r \diff z.
\end{multline}
Taking one half the sum of \eqref{StepA} and \eqref{StepB}, we find that
\begin{multline} 
    \frac{\diff}{\diff t}
    \int_0^\infty \int_0^\infty
    r^2 \omega_\theta(r,z,t)\diff r \diff z
    =
    \frac{1}{2\pi} 
    \int_0^\infty \int_0^\infty
    \int_0^\infty \int_0^\infty 
    \int_0^1
    r\bar{r}\omega_\theta(r,z,t)
    \omega_\theta(\bar{r},\bar{z},t)
    \\
    \left(-\frac{(z+\Bar{z})
    s\left(1-s^2\right)^{-\frac{1}{2}}}
    {\left((z+\Bar{z})^2+r^2+\Bar{r}^2
    -2r\Bar{r}s\right)^\frac{3}{2}}
    +\frac{(z+\Bar{z})
    s\left(1-s^2\right)^{-\frac{1}{2}}}
    {\left((z+\Bar{z})^2+r^2+\Bar{r}^2
    +2r\Bar{r}s\right)^\frac{3}{2}}
    \right)
    \diff s \diff\bar{r} \diff\bar{z} 
    \diff r \diff z.
\end{multline}
Distributing through the negative sign, this completes the proof.
\end{proof}

\begin{lemma} \label{LowerBoundLemma}
Suppose $\Omega_+ \subset \Pi^+$ is compact and bounded away from the axis $r=0$ and the plane $z=0$.
Then we have
\begin{multline}
    \kappa:= \frac{1}{2\pi}
    \inf_{(r,z),(\bar{r},\bar{z})\in \Omega_+}
    \int_0^1 
    \frac{(z+\Bar{z})
    s\left(1-s^2\right)^{-\frac{1}{2}}}
    {r\bar{r}} \\
    \left(\frac{1}
    {\left((z+\Bar{z})^2+r^2+\Bar{r}^2
    -2r\Bar{r}s\right)^\frac{3}{2}}
    -\frac{1}
    {\left((z+\Bar{z})^2+r^2+\Bar{r}^2
    +2r\Bar{r}s\right)^\frac{3}{2}}
    \right)
    \diff s
    >
    0.
\end{multline}
\end{lemma}

\begin{proof}
First we will introduce a function for the integrand, letting
\begin{multline}
    g(r,z,\bar{r},\bar{z},s)
    =
    \frac{(z+\Bar{z})
    s\left(1-s^2\right)^{-\frac{1}{2}}}
    {r\bar{r}} \\
    \left(\frac{1}
    {\left((z+\Bar{z})^2+r^2+\Bar{r}^2
    -2r\Bar{r}s\right)^\frac{3}{2}}
    -\frac{1}
    {\left((z+\Bar{z})^2+r^2+\Bar{r}^2
    +2r\Bar{r}s\right)^\frac{3}{2}}
    \right).
\end{multline}
We can see that this integrand is strictly positive; for all 
$(r,z),(\bar{r},\bar{z})\in \Omega_+$
and for all $0<s<1$,
\begin{equation} \label{gPositive}
    g(r,z,\bar{r},\bar{z},s)>0.
\end{equation}
This implies that for all 
$(r,z),(\bar{r},\bar{z})\in \Omega_+$,
\begin{equation}
    \int_0^1 g(r,z,\bar{r},\bar{z},s) \diff s
    >
    \int_\frac{1}{3}^\frac{2}{3}
    g(r,z,\bar{r},\bar{z},s) \diff s,
\end{equation}
and so it suffices to show that
\begin{equation}
    \inf_{(r,z),(\bar{r},\bar{z})\in \Omega_+}
    \int_\frac{1}{3}^\frac{2}{3}
    g(r,z,\bar{r},\bar{z},s) \diff s
    >
    0.
\end{equation}
Let
\begin{equation}
    G(r,z,\bar{r},\bar{z})
    =
    \int_\frac{1}{3}^\frac{2}{3}
    g(r,z,\bar{r},\bar{z},s) \diff s.
\end{equation}
The positivity of $g$ in \eqref{gPositive} implies that for all 
$(r,z),(\bar{r},\bar{z})\in\Omega_+$,
\begin{equation}
    G(r,z,\bar{r},\bar{z})>0.
\end{equation}

It is also simple to check that $g$ is a uniformly continuous function on 
$\Omega_+ \times \Omega_+ \times 
\left[\frac{1}{3},\frac{2}{3}\right]$,
and therefore $G$ is a uniformly continuous function on $\Omega_+\times \Omega_+$. A continuous function must attain its infimum on a compact set, so this implies that
\begin{equation}
    \inf_{(r,z),(\bar{r},\bar{z})\in \Omega_+}
    G(r,z,\bar{r},\bar{z})>0.
\end{equation}
This completes the proof.
\end{proof}

We are now ready to establish finite-time blowup for smooth solutions of the inviscid vortex stretching equation; we will prove Theorem \ref{BlowupThmIntro}, which is restated here for the reader's convenience.

\begin{theorem} \label{BlowupThm}
Suppose $\omega^0\in H^s_*, s>\frac{3}{2}$ is not identically zero, and further that $\omega^0_{\theta}(r,z)$ is odd in $z$, compactly supported away from the axis $r=0$ and the plane $z=0$,
and for all $r,z\geq 0$,
\begin{equation}
    \omega_\theta(r,z)\leq 0.
\end{equation}
Then the solution of the inviscid vortex stretching equation with this initial data blows up in finite-time
\begin{equation}
    T_{max}\leq 
    \frac{1}{\kappa Q^0},
\end{equation}
where $\kappa>0$ is a function of $\Omega_+$ defined as in Lemma \ref{LowerBoundLemma} and
\begin{equation}
    Q^0
    =
    -\int_0^\infty\int_0^\infty
    r^2\omega^0_{\theta}(r,z)
    \diff r \diff z.
\end{equation}
\end{theorem}

\begin{proof}
We will begin by letting
\begin{equation}
    Q(t)
    =
    -\int_0^\infty\int_0^\infty
    r^2\omega_{\theta}(r,z,t)
    \diff r \diff z.
\end{equation}
Applying Proposition \ref{DerivativeProp}
and Lemma \ref{LowerBoundLemma},
we can see that
\begin{align}
    \frac{\diff Q}{\diff t}
    &\geq 
    \kappa \int_0^\infty \int_0^\infty
    \int_0^\infty \int_0^\infty
    r^2\bar{r}^2
    \omega_{\theta}(r,z,t)
    \omega_{\theta}(\bar{r},\bar{z},t)
    \diff r\diff z \diff\bar{r} \diff\bar{z} \\
    &=
    \kappa\left(
    \int_0^\infty \int_0^\infty
    r^2 \omega_{\theta}(r,z,t)
    \diff r\diff z\right)^2 \\
    &=
    \kappa Q^2.
\end{align}
Note that Lemma \ref{OddLemma} guarantees that $\omega_\theta(\cdot,t)$ remains odd in $z$ and, crucially, remains non-positive on the upper half plane. This is essential because it guarantees that
\begin{equation}
    \omega_\theta(r,z)
    \omega_\theta(\bar{r},\bar{z})\geq 0.
\end{equation}

Integrating this differential inequality, we find that for all $0<t<T_{max}$
\begin{equation}
    Q(t)
    \geq 
    \frac{Q^0}{1-\kappa Q^0 t}.
\end{equation}
The assumption of compactly supported vorticity implies that
\begin{equation}
    Q(t)
    \leq 
    C\|\omega_\theta(\cdot,t)\|_{L^\infty},
\end{equation}
where $C>0$ is a constant independent of time depending only on $\Omega^0_+$.
Therefore, for all $0<t<T_{max}$,
\begin{equation}
    \|\omega(\cdot,t)\|_{L^\infty}
    \geq 
    \frac{C}
    {1-\kappa Q^0 t}.
\end{equation}
This implies that
\begin{equation}
    T_{max}\leq \frac{1}{\kappa Q^0},
\end{equation}
and this completes the proof.
\end{proof}

\begin{remark}
The functional $Q(t)$ is precisely the same used by Choi and Jeong in \cite{ChoiJeong} to show that axisymmetric, swirl-free solutions of the Euler equation with this colliding vortex ring geometry exhibit vorticity growth at infinity of the form
\begin{equation}
    \|\omega(\cdot,t)\|_{L^\infty}
    \geq 
    C(1+t)^{\frac{1}{15}-\epsilon}.
\end{equation}
For solutions of the Euler equation with the same geometric constraints as Theorem \ref{BlowupThm},
we have
    \begin{multline} \label{GrowthB}
        \frac{\diff}{\diff t}
        \left(-\int_0^\infty \int_0^\infty
        r^2 \omega_\theta(r,z,t)\diff r \diff z\right)
        = 
    \frac{1}{\pi} 
    \int_0^\infty \int_0^\infty
    \int_0^\infty \int_0^\infty 
    \int_0^1
    r\bar{r}\omega_\theta(r,z,t)
    \omega_\theta(\bar{r},\bar{z},t)
    \\
    \left(\frac{(z+\Bar{z})
    s\left(1-s^2\right)^{-\frac{1}{2}}}
    {\left((z+\Bar{z})^2+r^2+\Bar{r}^2
    -2r\Bar{r}s\right)^\frac{3}{2}}
    -\frac{(z+\Bar{z})
    s\left(1-s^2\right)^{-\frac{1}{2}}}
    {\left((z+\Bar{z})^2+r^2+\Bar{r}^2
    +2r\Bar{r}s\right)^\frac{3}{2}}
    \right)
    \diff s \diff\bar{r} \diff\bar{z} 
    \diff r \diff z.
    \end{multline}
This is the same exact equation for the growth of the functional $Q$, except that it is larger by a factor of $2$ for the actual Euler equation.
This difference of a factor of $2$ is due to an integration by parts involving the derivative in the radial direction.

The reason that there is not blowup in finite-time for the full axisymmetric, swirl-free Euler equation, is that the advection works to deplete the nonlinearity.
For the axisymmetric, swirl-free solutions of the inviscid vortex stretching equation, the support of the vorticity is preserved by the dynamics. For the full Euler equation, the flow is hyperbolic, compressing in $z$ and expanding in $r$. This introduces an anisotropy, with the support of the vorticity concentrating near the plane $z=0$ and getting stretched out in the radial direction.
This can prevent finite-time blowup by depleting the nonlinearity, because a key step to finite-time blowup for the inviscid vortex stretching equation, Lemma \ref{LowerBoundLemma}, relies on the support of the vorticity being uniformly bounded and uniformly bounded away from the axis $r=0$ and the plane $z=0$. For axisymmetric, swirl-free solutions of the Euler equation, even if we start with a vorticity that is bounded and bounded away from $r=0$ and $z=0$, the support of the vorticity will be made more and more anisotropic by the advection, so that 
\begin{equation}
    \liminf_{t\to +\infty} \kappa(t)=0.
\end{equation}
In fact we will prove something stronger.
\end{remark}

\begin{theorem} \label{AnisoThm}
Suppose $u\in C\left([0,+\infty);
H^s_{df}(\mathbb{R}^3)\right), s>\frac{5}{2}$
is an axisymmetric, swirl-free solution of the Euler equation and is not identically zero.
Suppose further that
$\omega^0_{\theta}(r,z)$ is odd in $z$, compactly supported away from the axis $r=0$ and the plane $z=0$, and for all $r,z \geq 0$,
\begin{equation}
    \omega_\theta(r,z)\leq 0.
\end{equation}
Then the support of the vorticity must develop anisotropies to the extent that 
\begin{equation}
    \int_0^\infty \kappa(t) \diff t
    \leq 
    \frac{1}{2 Q^0},
\end{equation}
where that $Q^0$ is defined as in Theorem \ref{BlowupThm} and
\begin{multline}
    \kappa(t):= \frac{1}{2\pi}
    \inf_{(r,z),(\bar{r},\bar{z})\in \Omega_+(t)}
    \int_0^1 
    \frac{(z+\Bar{z})
    s\left(1-s^2\right)^{-\frac{1}{2}}}
    {r\bar{r}} \\
    \left(\frac{1}
    {\left((z+\Bar{z})^2+r^2+\Bar{r}^2
    -2r\Bar{r}s\right)^\frac{3}{2}}
    -\frac{1}
    {\left((z+\Bar{z})^2+r^2+\Bar{r}^2
    +2r\Bar{r}s\right)^\frac{3}{2}}
    \right)
    \diff s,
\end{multline}
where $\Omega_+(t)$ is the support of $\omega_\theta(\cdot,t)$ in the upper half space.
\end{theorem}

\begin{proof}
As in Theorem \ref{BlowupThm}, 
we will let
\begin{equation}
    Q(t)
    =
    -\int_0^\infty\int_0^\infty
    r^2\omega_{\theta}(r,z,t)
    \diff r \diff z.
\end{equation}
From there we can compute that
\begin{equation}
    \frac{\diff Q}{\diff t}
    =
    \int_0^\infty
    \int_0^\infty 
    r^2 \left(
    \partial_r(u_r(r,z,t)\omega_\theta(r,z,t))
    +\partial_z(u_z(r,z,t)\omega_\theta(r,z,t))
    \right) \diff r \diff z.
\end{equation}
Integrating by parts, we find that
\begin{equation}
    \frac{\diff Q}{\diff t}
    =-2\int_0^\infty
    \int_0^\infty 
    r u_r(r,z,t)\omega_{\theta}(r,z,t)
    \diff r \diff z,
\end{equation}
noting that the geometric properties of our initial data are preserved by the dynamics and guarantee that the boundary terms vanish.
Applying Proposition \ref{DerivativeProp}, we can see that for all $0<t<+\infty$,
\begin{multline}
    \frac{\diff Q}{\diff t}
    =
    \frac{1}{\pi} 
    \int_0^\infty \int_0^\infty
    \int_0^\infty \int_0^\infty 
    \int_0^1
    r\bar{r}\omega_\theta(r,z,t)
    \omega_\theta(\bar{r},\bar{z},t)
    \\
    \left(\frac{(z+\Bar{z})
    s\left(1-s^2\right)^{-\frac{1}{2}}}
    {\left((z+\Bar{z})^2+r^2+\Bar{r}^2
    -2r\Bar{r}s\right)^\frac{3}{2}}
    -\frac{(z+\Bar{z})
    s\left(1-s^2\right)^{-\frac{1}{2}}}
    {\left((z+\Bar{z})^2+r^2+\Bar{r}^2
    +2r\Bar{r}s\right)^\frac{3}{2}}
    \right)
    \diff s \diff\bar{r} \diff\bar{z} 
    \diff r \diff z,
\end{multline}
and consequently,
\begin{align}
    \frac{\diff Q}{\diff t}(t)
    &\geq
    2\kappa(t)
    \int_0^\infty \int_0^\infty
    \int_0^\infty \int_0^\infty 
    r^2\bar{r}^2\omega_\theta(r,z,t)
    \omega_\theta(\bar{r},\bar{z},t)
    \diff\bar{r} \diff\bar{z} 
    \diff r \diff z \\
    &=
    2\kappa(t)\left(
    \int_0^\infty \int_0^\infty 
    r^2 \omega_\theta(r,z,t)
    \diff r \diff z \right)^2 \\
    &=
    2\kappa(t) Q(t)^2.
\end{align}

Integrating this differential inequality we find that for all $0<t<+\infty$,
\begin{equation}
    Q(t)\geq \frac{Q^0}{1-2Q^0
    \int_0^t \kappa(\tau)\diff\tau}.
\end{equation}
Choi and Jeong proved \cite{ChoiJeong} an upper bound on the growth of $Q(t)$ like
\begin{equation}
    Q(t)\leq C(1+t)^2,
\end{equation}
so we may conclude that for all $0<t<+\infty$,
\begin{equation}
    \int_0^t \kappa(\tau) \diff\tau <\frac{1}{2Q^0}.
\end{equation}
We know $\kappa(t)>0$ for all $0\leq t<+\infty$,
and so this immediately implies that
\begin{equation}
    \int_0^\infty \kappa(t) \diff t
    \leq \frac{1}{2Q^0},
\end{equation}
and this completes the proof.
\end{proof}

\section*{Acknowledgements}

This work was supported by the Pacific Institute for the Mathematical Sciences, while the author was a PIMS postdoctoral fellow at the University of British Columbia.

\bibliographystyle{plain}
\bibliography{bib}

\end{document}